\pdfoutput=1
\RequirePackage{ifpdf}
\ifpdf % We are running pdfTeX in pdf mode
\documentclass[pdftex]{sigma}
\else
\documentclass{sigma}
\fi

\usepackage{subfig}
\usepackage{tikz,mathtools}

\numberwithin{equation}{section}

\newtheorem{thm}{Theorem}[section]
\newtheorem{lem}[thm]{Lemma}
\newtheorem{prop}[thm]{Proposition}
\newtheorem{cor}[thm]{Corollary}

\theoremstyle{definition}

\newtheorem{rem}[thm]{Remark}

\newcommand{\U}{\mathcal{U}}
\newcommand{\Q}{\mathbb{Q}}
\newcommand{\R}{\mathbb{R}}
\newcommand{\C}{\mathbb{C}}
\newcommand{\Z}{\mathbb{Z}}
\newcommand{\T}{\mathbb{T}}
\newcommand{\id}{\mathsf{id}}
\newcommand{\mc}{\mathcal}
\newcommand{\mf}{\mathfrak}
\newcommand{\inner}[1]{\left<\smash[t]{#1}\right>}
\newcommand{\de}{\mathrm{d}}
\newcommand{\lbrak}{[\hspace{-0.5pt}[}
\newcommand{\rbrak}{]\hspace{-0.5pt}]}
\newcommand{\ga}{\mf{p}}
\newcommand{\gb}{\mf{q}}
\newcommand{\gc}{\mf{t}}
\renewcommand{\hbar}{\nu}

\begin{document}

\allowdisplaybreaks

\newcommand{\arXivNumber}{1509.05347}

\renewcommand{\PaperNumber}{015}

\FirstPageHeading

\ShortArticleName{Non-Associative Geometry of Quantum Tori}
\ArticleName{Non-Associative Geometry of Quantum Tori}

\Author{Francesco D'ANDREA~$^{\dag\ddag}$ and Davide FRANCO~$^\ddag$}
\AuthorNameForHeading{F.~D'Andrea and D.~Franco}

\Address{$^\dag$~Dipartimento di Matematica e Applicazioni, Universit\`a di Napoli ``Federico II'',\\
 \hphantom{$^\dag$}~Complesso MSA, Via Cintia, 80126 Napoli, Italy}
\Address{$^\ddag$~I.N.F.N.~Sezione di Napoli, Complesso MSA, Via Cintia, 80126 Napoli, Italy}
\Email{\href{mailto:francesco.dandrea@unina.it}{francesco.dandrea@unina.it}, \href{mailto:dfranco@unina.it}{dfranco@unina.it}}

\ArticleDates{Received October 02, 2015, in f\/inal form February 04, 2016; Published online February 07, 2016}

\Abstract{We describe how to obtain the imprimitivity bimodules of the noncommutative torus from a ``principal bundle'' construction, where the total space is a quasi-associative deformation of a $3$-dimensional Heisenberg manifold.}


\Classification{58B34; 46L87; 53D55}

\section{Introduction}

In dif\/ferential geometry a standard way to construct vector bundles is from a principal bundle and a representation of the structure group. This construction works in noncommutative geometry too \cite{Con94,GVF01,Lan02}, with spaces replaced by algebras, vector bundles replaced by f\/initely generated projective modules, structure groups replaced by Hopf algebras or compact quantum groups, and principal bundles replaced by algebra extensions with suitable additional properties (see, e.g.,~\cite{BM92,Haj96}). When the structure group is~$U(1)$, it is possible to reconstruct the total space of the ``bundle'' (more precisely, a strongly graded $C^*$-algebra) from the base space and a~noncommutative ``line bundle'' (a~self-Morita equivalence bimodule), cf.~\cite{AKL14} (see also~\cite{ADL15,D15}).

A case study is provided by the $C^*$-algebra of the noncommutative torus $A_\theta$, $\theta\in\R{\setminus}\Q$ \cite{Con80,Rie81}.
The group ${\rm SL}_2(\Z)$ acts on $\R{\setminus}\Q$ by fractional linear transformations,
\begin{gather}\label{eq:gtheta}
g\theta :=\frac{a\theta+b}{c\theta+d}
 ,\qquad\forall\,\theta\in\R{\setminus}\Q
\qquad\text{and}\qquad g=
 \begin{pmatrix} a & b \\ c & d \end{pmatrix}
\in {\rm SL}_2(\Z) ,
\end{gather}
and for every $g$ as above there is a full right Hilbert $A_\theta$-module $E_g(\theta)$ which is a Morita equiva\-len\-ce bimodule between $A_\theta$ and the algebra $\operatorname{End}_{A_\theta}(E_g(\theta))\simeq A_{g\theta}$ \cite{Con80,Rie83}.
The equivalence class of the right module~$E_g(\theta)$ only depends on the `degree' and `rank' of the module, def\/ined by
$\deg (E_g(\theta))=c$ and $\operatorname{rank}(E_g(\theta))=c\theta+d$ (see, e.g.,~\cite{PS02}). Every other f\/initely generated projective right $A_\theta$-module is a direct sum of such modules.

In this paper we show how one can formally derive some of these bimodules~-- those with rank~$1$ for~$\theta\to 0$ (hence, line bundles)~-- from some kind of $U(1)$-bundle, but at the price of working with quasi-associative algebras, i.e., monoids in the category of representations of a~quasi-Hopf algebra.
Several well-known properties of these bimodules can be shown to be a~consequence of the explicit form of the coassociator (cf.~Lemma~\ref{gcl}), including the associativity property of the pairing of bimodules of Polishchuk and Schwarz~\cite{PS02}.

The plan of the paper is the following. In Section~\ref{sec:2} we recall some basic def\/initions about twist deformations of quasi-Hopf algebras and algebra modules. In Section~\ref{sec:3} we will introduce the twist we are interested in, based on the universal enveloping algebra of the $3$-dimensional Heisenberg Lie algebra $\mf{h}_3(\R)$; we will also prove some generalized associativity property of twisted module algebras, namely Proposition~\ref{prop:gcl}. In Section~\ref{sec:4} we will give a concrete realization of $\mf{h}_3(\R)$ in terms of dif\/ferential operators on the total space of a principal $U(1)$-bundle $M_3\to\T^2$ on the $2$-torus, and apply to it the deformation recipe of Section~\ref{sec:2} with the twist introduced in Section~\ref{sec:3}: the total space becomes a quasi-associative $\Z$-graded algebra, which in degree $0$ is the algebra of smooth functions on the noncommutative torus; in degree $n\neq 0$ we get bimodules that we compare with Connes--Rief\/fel imprimitivity bimodules.
In Section~\ref{sec:5} we give a slightly dif\/ferent version of the construction, more natural if one is interested in complex structures: we show that, in this case, the twisted algebra of functions on the Heisenberg manifold has an associative commutative subalgebra given by ordinary theta functions.
In Section~\ref{sec:6} we study vector bundles of any rank.

{\bf Notations.} By a \emph{algebra} we shall always mean a unital algebra (not necessarily associative nor commutative) over a commutative unital ring $R$; the algebraic tensor product over $R$ will be denoted by $\otimes_R$, or simply $\otimes$ if there is no risk of confusion;
by a Hopf algebra over $\C\lbrak\hbar\rbrak$ we shall always mean a \emph{topological} Hopf algebra, completed in the $h$-adic topology, and by $\otimes_{\C\lbrak\hbar\rbrak}$ the completed tensor product \cite[Section~4.1A]{CP94}.

\section{Mathematical preliminaries}\label{sec:2}

\begin{figure}[t]
 \centering
 \subfloat[Pentagon diagram]{
\begin{small}
\begin{tikzpicture}[scale=0.85]
\tikzset{>=stealth}

\node (P0) at (90:2.8cm) {$((A\otimes B)\otimes C)\otimes D$};
\node (P1) at (90+72:2.5cm) {$(A\otimes B)\otimes (C\otimes D)$};
\node (P2) at (90+2*72:2.5cm) {$\mathllap{A\otimes (B\,\otimes\,}(C\otimes D))$};
\node (P3) at (90+3*72:2.5cm) {$A\otimes ((B \mathrlap{\,\otimes\,C)\otimes D)}$} ;
\node (P4) at (90+4*72:2.5cm) {$(A\otimes (B\otimes C))\otimes D$};

\draw
(P0) edge[->] node[left=1pt] {\raisebox{8pt}{\footnotesize $\Phi_{A\otimes B,C,D}$}} (P1)
(P1) edge[->] node[left=1pt] {{\footnotesize $\Phi_{A,B,C\otimes D}$}} (P2)
(P3) edge[->] node[above] {{\footnotesize $\Phi_{B,C,D}$}} (P2)
(P4) edge[->] node[right=1pt] {{\footnotesize $\Phi_{A,B\otimes C,D}$}} (P3)
(P0) edge[->] node[right=1pt] {\raisebox{8pt}{\footnotesize $\Phi_{A,B,C}$}} (P4);
\end{tikzpicture}
\end{small}
 \label{fig:P} }
	\hspace{3mm}
 \subfloat[Triangle diagram]{
\begin{small}
\begin{tikzpicture}
\tikzset{>=stealth}

\node (P0) at (0,0) {$(A\otimes I)\otimes B$};
\node (P1) at (3.6,0) {$A\otimes (I\otimes B)$};
\node (P2) at (1.8,-2) {$A\otimes B$};
\node (P3) at (1.8,-3) {};

\draw
(P0) edge[->] node[above] {{\footnotesize $\Phi_{A,I,C}$}} (P1)
(P0) edge[->] node[below left]{\raisebox{8pt}{\footnotesize $\rho_A\otimes\id_B$}} (P2)
(P1) edge[->] node[below right]{\raisebox{8pt}{\footnotesize $\id_A\otimes\lambda_B$}} (P2);
\end{tikzpicture}
\end{small}
 \label{fig:T} }
	\vspace{-10pt}\caption{}\vspace{-5pt}
\end{figure}

In this section, we recall some def\/initions and properties of quasi-Hopf algebras and Drinfeld twists, from~\cite{CP94,Dri90,Dri91,ML98}.

\subsection{Monoidal categories}\label{sec:2.1}

A \emph{monoidal category} is a category $\mathbf{C}$ equipped with a functor \mbox{$\otimes\colon \mathbf{C}\times\mathbf{C}\to\mathbf{C}$} which, modulo natural isomorphisms, is associative and unital \cite[Section~5.1]{CP94}. More precisely, there is an object~$I$ and three natural isomorphisms~-- $\Phi$ between the functors $(\,\rule{5pt}{0.5pt}\,\otimes\,\rule{5pt}{0.5pt}\,)\otimes\,\rule{5pt}{0.5pt} $~and~$ \rule{5pt}{0.5pt}\,\otimes (\,\rule{5pt}{0.5pt}\,\otimes\,\rule{5pt}{0.5pt}\,)$, $\lambda$~between~$I\otimes\,\rule{5pt}{0.5pt}$~and the identity, $\rho$~between~$\rule{5pt}{0.5pt}\,\otimes I$ and the identity~-- such that the diagrams~\ref{fig:P} and~\ref{fig:T} commute for all objects~$A$, $B$, $C$, $D$.
Examples are the category of modules over a~f\/ield (vector spaces), over a group or a Hopf algebra (representations), over a topological Hopf algebra (with completed tensor product).

\begin{figure}[t]
 \centering
 \subfloat[Associativity]{
\begin{small}
\begin{tikzpicture}[scale=0.7]
\tikzset{>=stealth}

\node (P0) at (-90:2.5cm) {$A$};
\node (P1) at (-90+72:2.5cm) {$A\otimes A$};
\node (P2) at (-90+2*72:2.5cm) {$A\otimes \mathrlap{(A\otimes A)}$};
\node (P3) at (-90+3*72:2.5cm) {$\mathllap{(A\otimes A)}\otimes A$} ;
\node (P4) at (-90+4*72:2.5cm) {$A\otimes A$};

\draw
(P1) edge[->] node[below right] {{\footnotesize $m$}} (P0)
(P2) edge[->] node[right=1pt] {{\footnotesize $\id\otimes m$}} (P1)
(P3) edge[->] node[above] {{\footnotesize $\Phi_{A,A,A}$}} (P2)
(P3) edge[->] node[left=1pt] {{\footnotesize $m\otimes\id$}} (P4)
(P4) edge[->] node[below left] {{\footnotesize $m$}} (P0);
\end{tikzpicture}
\end{small}
 \label{fig:monA} }
	\hspace{1.2cm}
 \subfloat[Unitality]{
\begin{small}
\begin{tikzpicture}
\tikzset{>=stealth}

\node (P0) at (0,0) {$I\otimes A$};
\node (P1) at (2,0) {$A\otimes A$};
\node (P2) at (4,0) {$A\otimes I$};
\node (P3) at (2,-2) {$A$};
\node (P4) at (1.8,-2.8) {};

\draw
(P0) edge[->] node[above] {{\footnotesize $\eta\otimes\id$}} (P1)
(P2) edge[->] node[above] {{\footnotesize $\id\otimes\eta$}} (P1)
(P0) edge[->] node[left] {{\footnotesize $\lambda_A$}} (P3)
(P1) edge[->] node[right] {{\footnotesize $m$}} (P3)
(P2) edge[->] node[right=1pt] {\raisebox{-7pt}{\footnotesize $\rho_A$}} (P3);
\end{tikzpicture}
\end{small}
 \label{fig:monB} }
	\vspace{-10pt}\caption{}\vspace{-5pt}
\end{figure}

A \emph{monoid} \cite[Section~VII.3]{ML98} in a monoidal category is an object $A$ together with two arrows \mbox{$m\colon A\otimes A\to A$} and $\eta\colon I\to A$, the ``multiplication'' and ``unit'',
satisfying the usual axioms of a~unital associative algebra modulo natural transformations, namely the diagrams~\ref{fig:monA} and~\ref{fig:monB} must commute. Monoids in the category of vector spaces are associative algebras, in the category of representations of a~Hopf algebra~$H$ are $H$-module algebras.

\subsection{Quasi-Hopf algebras}\label{sec:2.2}

Let $H$ be an algebra and $\Delta\colon H\to H\otimes H$ a linear map.
For all $n\geq m$ and all $1\leq i_1<i_2<\cdots<i_m\leq n$, we will denote by $h\mapsto h_{i_1\ldots i_m}$ the
linear map $H^{\otimes m}\to H^{\otimes n}$ def\/ined on homogeneous tensors
$h=a_1\otimes a_2 \otimes\cdots\otimes a_m$ as follows: we put $a_k$ in the leg
$i_k$ of the tensor product for all $k=1,\ldots, m$, and f\/ill the additional
$n-m$ legs with $1$. The subscript $(i_ki_{k+1})$ in parenthesis means
that we apply $\Delta$ to $a_k$ and put the f\/irst leg of $\Delta(a_k)$ in position~$i_k$ and the second in position $i_{k+1}$.
So, for example, if $h=a\otimes b$, $h_{13}=a\otimes 1\otimes b$,
$h_{(12)3}=\Delta(a)\otimes b$, etc.

A \emph{quasi-bialgebra} (over a commutative ring $R$) is an associative algebra $H$ together with two homomorphisms $\Delta\colon H\to H\otimes H$ and $\epsilon\colon H\to R$ (the ``coproduct'' and ``counit'') and an invertible element $\Phi\in H\otimes H\otimes H$ (the ``coassociator'') satisfying the following conditions (here we follow the notations of \cite{Dri90,Dri91}):
\begin{gather}
\Phi_{12(34)}\Phi_{(12)34} = \Phi_{234}\Phi_{1(23)4}\Phi_{123} ,\label{eq:pent} \\
(\id\otimes\epsilon\otimes\id )(\Phi)=1 ,\label{eq:counitality}
\end{gather}
and
\begin{gather}
(\id\otimes\Delta)\Delta(h) =\Phi
(\Delta\otimes\id)\Delta(h)\Phi^{-1} , \label{eq:quasicoass} \\
(\id\otimes\epsilon)\Delta(h) =h=(\epsilon\otimes\id)\Delta(h) ,\nonumber
\end{gather}
for all $h\in H$.
A \emph{quasi-Hopf algebra} is a quasi-bialgebra satisfying an additional condition which for $\Phi=1$ reduces to the existence and bijectivity of an antipode, see, e.g., \cite[Def\/inition~16.1.1]{CP94}.

Any Hopf algebra is a quasi-Hopf algebra with trivial coassociator, $\Phi=1$. In fact, in the Hopf case any $\Phi$ which is \emph{invariant} (i.e., commutes with the image of the iterated coproduct) does the job. If $\Phi$ is invariant,
\eqref{eq:quasicoass} reduces to the coassociativity condition $(\id\otimes\Delta)\Delta=(\Delta\otimes\id)\Delta$ and we get the usual def\/inition of bialgebra/Hopf algebra.

If $H$ is a commutative Hopf algebra, \eqref{eq:pent} can be interpreted as a $3$-cocycle condition in the Hopf algebra cohomology of $H$ \cite[Section~2.3]{Maj95}. In general, if $A$, $B$, $C$ are three $H$-modules, we may think of $\Phi$ as a module map $(A\otimes B)\otimes C\to A\otimes (B\otimes C)$, with the natural isomorphism of vector spaces understood. Condition \eqref{eq:pent} becomes the associativity condition~\ref{fig:P} for the category of $H$-modules, which is then monoidal with unit object $I=R$ given by the ground ring
(with module structure being given by the counit, so that~\eqref{eq:counitality} implies the commutativity of the diagram~\ref{fig:T}, and module structure on a tensor product def\/ined by the coproduct).

A monoid $A$ in this category is a $H$-module algebra, that is an algebra~$A$ with multiplication satisfying the quasi-associativity condition (cf.~diagram~\ref{fig:monA}):
\begin{gather}\label{eq:quasiassoc}
m(m\otimes\id)=m(\id\otimes m)\Phi ,
\end{gather}
Commutativity of the diagram~\ref{fig:monB} means that~$A$ is unital, with $1_A:=\eta(1_R)$. The condition that~$m$ and~$\eta$ are morphisms in the category, i.e., $H$-module maps, gives $h.1_A=\epsilon(h)1_A$ and
\begin{gather*}
h\circ m=m\circ \Delta(h) ,
\end{gather*}
for all $h\in H$. Equation~\eqref{eq:quasiassoc} becomes the usual associativity condition if the action of $\Phi$ on $A\otimes A\otimes A$ is trivial (a~suf\/f\/icient condition of course is that $\Phi=1$).

The above def\/initions remain valid if the ground ring is $R=\C\lbrak\hbar\rbrak$, with $\otimes_R$ the completed tensor product.

The advantage of quasi-bialgebras (resp.~quasi-Hopf algebras) over ordinary bialgebras (resp.\ Hopf algebras) is that there is a~``gauge action'' of the group of invertible elements in $H\otimes H$ that doesn't change the category of modules.

Let $H$ be a quasi-bialgebra (resp.~quasi-Hopf algebra) and $F\in H\otimes H$ an invertible element satisfying $(\id\otimes\epsilon)(F)=(\epsilon\otimes\id)(F)=1$. One can def\/ine a new quasi-bialgebra (resp.~quasi-Hopf algebra) $H_F$ given by $H$ as an algebra and with the same counit, but with a new coproduct $\Delta_F$ and coassociator $\Phi_F$ def\/ined by
\begin{gather*}
 \Delta_F(h) :=F\Delta(h)F^{-1},\qquad \forall\, h\in H,\\
 \Phi_F :=F_{23}F_{1(23)} \Phi
\big(F^{-1}\big)_{(12)3}\big(F^{-1}\big)_{12} .
\end{gather*}
It turns out that the categories of $H$-modules and $H_F$-modules are equivalent as monoidal categories.
In particular, if $(A,m)$ is a $H$-module algebra, there is a $H_F$-module algebra $(A,m_F)$ given by~$A$ as a vector space, with the same unit element, and with multiplication:
\begin{gather}\label{eq:mf}
m_F:=m\circ F^{-1} .
\end{gather}
We will refer to $H_F$ as a twist deformation of~$H$, and to~$F$ as a \emph{twisting element} based on~$H$.

\subsection{Quantum universal enveloping algebras}\label{sec:QUEA}

By a \emph{deformation} of a quasi-bialgebra over a f\/ield $\Bbbk$ we mean a topological quasi-bialgebra $H_\hbar$ over $\Bbbk\lbrak\hbar\rbrak$
such that $H_\hbar/\hbar H_\hbar\simeq H$ as quasi-bialgebras and $H_\hbar\simeq H\lbrak\hbar\rbrak$ as $\Bbbk\lbrak\hbar\rbrak$-modules.
If~$H$ is a quasi-Hopf algebra, any deformation as a quasi-bialgebra is in fact a quasi-Hopf algebra \cite[Section~16.1C]{CP94}.

If $H=\U(\mf{g})$ is the universal enveloping algebra of a Lie algebra $\mf{g}$, with standard Hopf algebra structure and trivial coassociator, a deformation $H_\hbar$ with coassociator \mbox{$\Phi_\hbar\equiv 1$ mod $\hbar^2$} will be called a \emph{quantum universal enveloping algebra} (or~QUEA). Any twist of $\U(\mf{g})$ by a twisting element $F_\hbar$ based on $\U(\mf{g})\lbrak\hbar\rbrak$
(and satisfying \mbox{$F_\hbar\equiv 1$ mod $\hbar$}) is a~QUEA, and roughly speaking every QUEA arises in this way (see Theorem~16.1.11 of~\cite{CP94} for the precise statement).

\section{A twist based on the Heisenberg Lie algebra}\label{sec:3}

The example which is of interest to us is based on the Lie algebra $\mf{h}_3(\R)$ of the $3$-dimensional Heisenberg group, or more precisely on the Hopf algebra
$\U(\mf{h}_3(\R))$.

The Lie algebra $\mf{h}_3(\R)$ has a basis of three elements $\ga$, $\gb$, $\gc$, with $\gc$ central and
\begin{gather}\label{eq:h3R}
[\ga,\gb]=\gc .
\end{gather}
Let us call \emph{parameter space}
$\Theta\subset\C\lbrak\hbar\rbrak$ the ideal:
\begin{gather*}
\Theta:=\hbar \C\lbrak\hbar\rbrak .
\end{gather*}
Note that for all $\{a_n\}_{n\geq 0}$ belonging to a complex vector space $A$ and for all $\theta\in\Theta$, the expression $\sum_{n\geq 0}a_n\theta^n$ is a well def\/ined element of $A\lbrak\hbar\rbrak$ (for each $N\geq 0$, the coef\/f\/icient of $\hbar^N$ is a f\/inite linear combination
of elements of~$A$).

There is an action $\alpha$ of $\Z$ by automorphisms on the vector space $\Theta$ given by{\samepage
\begin{gather}\label{eq:nact}
\alpha_k(\theta)=\frac{\theta}{1+k\theta}=
\theta-k\theta^2+\cdots
 ,\qquad k\in\Z,\quad \theta\in\Theta.
\end{gather}
Clearly $\alpha$ maps $\Theta$ into itself, and $\alpha_j\alpha_k=\alpha_{j+k}$ $\forall\, j,k\in\Z$.}

For all $\theta\in\Theta$, a twisting element based on $\U(\mf{h}_3(\R))\lbrak \hbar\rbrak$ is
\begin{gather*}
F_\theta=\exp\big\{\theta\hspace{1pt}\ga\hspace{1pt}\otimes\gb\big\} .
\end{gather*}
Let $A$ be any $\U(\mf{h}_3(\R))$-module algebra (associative, since the coassociator is trivial), and for all $n\in\Z$ def\/ine
\begin{gather}\label{eq:An}
A_n:=\big\{ a\in A\colon \gc.a=na \big\} .
\end{gather}
Since $\gc$ is central, each $A_n$ is a $\U(\mf{h}_3(\R))$-module itself, and $A_\bullet:=\bigoplus_{n\in\Z}A_n$ is a graded (associative) module subalgebra of~$A$.

We let $a\ast_\theta b:=m\circ F_\theta^{-1}(a\otimes b)$ be the product on $A\lbrak\hbar\rbrak$ def\/ined in \eqref{eq:mf}, and denote by $A^\theta=(A\lbrak\hbar\rbrak,\ast_\theta)$ the deformed algebra.
Since~$\gc$ is central, each set $A_n\lbrak\hbar\rbrak$ is stable under the action of~$\ga$,~$\gb$. So
$A^\theta_\bullet=(A_\bullet\lbrak\hbar\rbrak,\ast_\theta)$ is a $\Z$-graded subalgebra of $A^\theta$, and $A^\theta_0=(A_0\lbrak\hbar\rbrak,\ast_\theta)$ is a~yet smaller subalgebra.

\begin{lem}\label{gcl}
For all $\theta,\theta'\in\Theta$:
\begin{gather}\label{eq:gcl}
(\Delta\otimes\id)\big(F_\theta^{-1}\big)\big(F_{\theta'}^{-1}\otimes 1\big)
=(\id\otimes\Delta)\big(F_{\theta'}^{-1}\big)\big(1\otimes F_\theta^{-1}\big)\Phi_{\theta,\theta'},
\end{gather}
where
\begin{gather}\label{eq:Phittp}
\Phi_{\theta,\theta'}:=\exp\big\{{-}\ga\otimes\left(\theta-\theta'-\theta\theta' \gc\right)\otimes\gb\big\} .
\end{gather}
\end{lem}

\begin{proof}
From \eqref{eq:gcl}, recalling that $\ga$, $\gb$ are primitive elements
\begin{gather}
\Phi_{\theta,\theta'} =
(1\otimes F_\theta)
(\id\otimes\Delta)(F_{\theta'})(\Delta\otimes\id)\big(F_\theta^{-1}\big)\big(F_{\theta'}^{-1}\otimes 1\big) \nonumber \\
\hphantom{\Phi_{\theta,\theta'} }{}
=e^{\theta 1\otimes \ga\otimes\gb}
e^{\theta'( \ga\otimes 1\otimes\gb+\ga\otimes\gb\otimes 1 )}
e^{-\theta ( \ga\otimes 1\otimes\gb+1\otimes\ga\otimes\gb)}
e^{-\theta' \ga\otimes\gb\otimes 1} . \label{eq:3.3}
\end{gather}
Next, we use Baker--Campbell--Hausdorf\/f formula, which for two elements $X$, $Y$ of an associative algebra with central commutator $[X,Y]$ reduces to $e^{\hbar X}e^{\hbar Y}=e^{\hbar^2[X,Y]}e^{\hbar Y}e^{\hbar X}$. We get
\begin{gather*}
e^{-\theta ( \ga\otimes 1\otimes\gb+1\otimes\ga\otimes\gb)}
e^{-\theta' \ga\otimes\gb\otimes 1}=
e^{\theta\theta' \ga\otimes\gc\otimes\gb}
e^{-\theta' \ga\otimes\gb\otimes 1}
e^{-\theta ( \ga\otimes 1\otimes\gb+1\otimes\ga\otimes\gb)} .
\end{gather*}
Using this in \eqref{eq:3.3} we get
\begin{gather*}
\Phi_{\theta,\theta'} =
e^{\theta\theta' \ga\otimes\gc\otimes\gb}
e^{\theta 1\otimes \ga\otimes\gb}
 (e^{\theta'( \ga\otimes 1\otimes\gb+\ga\otimes\gb\otimes 1 )}
e^{-\theta' \ga\otimes\gb\otimes 1})
e^{-\theta ( \ga\otimes 1\otimes\gb+1\otimes\ga\otimes\gb)} \\
\hphantom{\Phi_{\theta,\theta'} }{}
=e^{\theta\theta' \ga\otimes\gc\otimes\gb}
e^{\theta 1\otimes \ga\otimes\gb}
e^{\theta' \ga\otimes 1\otimes \gb}
e^{-\theta ( \ga\otimes 1\otimes\gb+1\otimes\ga\otimes\gb)} .
\end{gather*}
Now all the exponents mutually commute, and after some simplif\/ication we arrive at~\eqref{eq:Phittp}.
\end{proof}

In particular, we can observe that the coassociator:
\begin{gather}\label{eq:coassociator}
\Phi_{\theta,\theta}=e^{\theta^2\ga \otimes \gc \otimes \gb}
\end{gather}
of the QUE algebra is not trivial nor invariant, so that we are dealing with a genuine quasi-Hopf deformation of $\U(\mf{h}_3(\R))$.
Many properties of $A^{\theta}$ can be deduced from Lemma~\ref{gcl}.

\begin{prop}\label{prop:gcl}
If $\theta'=\alpha_n(\theta)$, with $\alpha$ as in \eqref{eq:nact}, one has the generalized associativity law:
\begin{gather}\label{eq:genass}
(a\ast_{\theta'}b)\ast_\theta c=
a\ast_{\theta'}(b\ast_\theta c)
\qquad \forall\, a,c\in A\lbrak\hbar\rbrak,\quad b\in A_n\lbrak\hbar\rbrak.
\end{gather}
\end{prop}
\begin{proof}
If $\theta'=\alpha_n(\theta)$,
from the def\/inition of the star product and the observation that $\gc.b=nb$, we deduce that $\Phi_{\theta,\theta'}$ is the identity on $a\otimes b\otimes c$ and then~\eqref{eq:genass} holds.
\end{proof}

\begin{cor}\label{cor:3.3} \quad
\begin{enumerate}\itemsep=0pt
\item[$1)$]
$A_0^{\theta}$ is an associative unital subalgebra of $A^{\theta}$.

\item[$2)$]
$E_n:=A_n\lbrak\hbar\rbrak$ is an $A_0^{\alpha_n(\theta)}$-$A_0^{\theta}$-bimodule $($with left module structure given by $\ast_{\alpha_n(\theta)}$ and right module structure given by $\ast_\theta)$.

\item[$3)$]
$ \ast_\theta\colon E_j\otimes E_k\to E_{j+k}$ descends to a map
$E_j\otimes_{\smash[t]{A_0^{\theta}}}\! E_k\to E_{j+k}$ $($where the left and right module structure are given by the $\ast_\theta$ multiplication$)$.

\item[$4)$]
For all $m,n,p\in\Z$ and
$\theta'=\alpha_n(\theta)$, the following diagram commutes:
\begin{center}
\begin{tikzpicture}
\tikzset{>=stealth, shorten >=1pt, shorten <=1pt}

\node (A) at (0,0) {$E_m\otimes_{\smash[t]{A_0^{\theta'}}}\! E_n\otimes_{\smash[t]{A_0^{\theta}}} E_p$};
\node (B) at (-1.9,-1.9) {$E_m\otimes_{\smash[t]{A_0^{\theta'}}}\! E_{n+p}$};
\node (C) at (1.9,-1.9) {$E_{m+n}\otimes_{\smash[t]{A_0^{\theta}}} E_p$};
\node (D) at (0,-3.8) {$E_{m+n+p}$};

\draw[->] (A) --node[left]{\raisebox{9pt}{\footnotesize $\id\otimes\ast_\theta$}} (B);
\draw[->] (A) --node[right]{\raisebox{9pt}{\footnotesize $\ast_{\theta'}\otimes\id$}} (C);
\draw[->] (B) --node[below left=-3pt]{{\footnotesize $\ast_{\theta'}\;$}} (D);
\draw[->] (C) --node[below right=-2pt]{{\footnotesize $\;\ast_\theta$}} (D);

\end{tikzpicture}
\end{center}
\end{enumerate}
\end{cor}

\section{Line bundles on the noncommutative torus}\label{sec:4}

Let $H_3(\R)$ be the group of upper triangular matrices
\begin{gather*}
(x,y,t):=\left[ \begin{matrix}
1 & x & t \\ 0 & 1 & y \\ 0 & 0 & 1
\end{matrix} \right] ,\qquad x,y,t\in\R.
\end{gather*}
With a slight abuse of notations, we identify elements of $\U(\mf{h}_3(\R))$ with left invariant vector f\/ields on $H_3(\R)$, that are generated by the dif\/ferential operators
\begin{gather*}%\label{eq:pqt}
\ga:=\frac{1}{\sqrt{2\pi}}\frac{\partial}{\partial x}
,\qquad
\gb:=\frac{-i}{\sqrt{2\pi}}\left(\frac{\partial}{\partial y}+x\frac{\partial}{\partial t}\right)
,\qquad
\gc:=\frac{1}{2\pi i}\frac{\partial}{\partial t}
.
\end{gather*}
The choice of normalization will be clear later on (one can check that \eqref{eq:h3R} is satisf\/ied).

Let $H_3(\Z):=\{(x,y,t)\in H_3(\R)\colon x,y,t\in\Z\}$.
By left invariance, the above vector f\/ields descend to the
$3$-dimensional Heisenberg manifold $M_3:=H_3(\Z)\backslash H_3(\R)$.
Thinking of functions on $K\backslash G$ as left $K$-invariant functions on $G$, we get
\begin{gather*}
C^\infty(M_3)=\big\{f\in C^\infty\big(\R\times\T^2\big)\colon f(x+1,y,t+y)=f(x,y,t)\big\} ,
\end{gather*}
where $\T^2=\R^2/\Z^2$ (so $f\in C^\infty(M_3)$ is periodic with period~$1$ in~$y$ and~$t$).
The action of central elements $(0,0,t)\in H_3(\R)$ descends to a principal action of $U(1)$ on $M_3$, and $U(1)\backslash M_3\simeq\T^2$. We identify
$C^\infty(\T^2)$ with the subset of $f\in C^\infty(M_3)$ that do not depend on~$t$ (and so are periodic in both $x$ and $y$):
\begin{gather*}
C^\infty\big(\T^2\big)=\big\{f\in C^\infty(M_3)\colon f(x,y,t+t')=f(x,y,t)\; \forall\,x,y,t,t'\big\} .
\end{gather*}
Let $A:=C^\infty(M_3)$. In the notation of previous section, $A_0=
C^\infty(\T^2)$ and, for $n\neq 0$, every element $f\in A_n$ can be written in the form
\begin{gather}\label{eq:fxy}
f(x,y,t)=\sum_{k\in\Z}\widetilde{f}\big(x+\tfrac{k}{n};k\big)e^{2\pi i(ky+nt)}
\end{gather}
for a unique Schwartz function $\widetilde{f}\colon \R\times\Z/n\Z\to\C$.
The bijection $A_n\to\mc{S}(\R\times\Z/n\Z)$, $f\mapsto\widetilde{f}$, is known as Weil--Brezin--Zak transform \cite[Section~1.10]{Fol89}.
Functions~\eqref{eq:fxy} can be interpreted as smooth sections of a non-trivial smooth line bundle on~$\T^2$ \cite{DFF13}.
The algebra $A_\bullet=\bigoplus_{n\in\Z}A_n$ is dense in~$A$ (in the uniform topology); indeed, by periodicity in~$t$, the only weight spaces of~$\gc$ appearing in the decomposition of~$A$ are those with integer weight.
It is \emph{strongly} $\Z$-graded, which is the algebraic counterpart of the principality of the bundle $M_3\to\T^2$ (see, e.g.,~\cite{ADL15} or~\cite{D15}).

By point~(1) of Corollary~\ref{cor:3.3}, $A_0^\theta$ is an associative subalgebra of $A^\theta$.
By standard Fourier analysis, it is not dif\/f\/icult to verify that $A^\theta_0$ is generated by two unitary elements, the functions
\begin{gather}\label{eq:VUtorus}
V(x,y,t):=e^{2\pi ix} ,\qquad
U(x,y,t):=e^{2\pi iy} ,
\end{gather}
with relation
\begin{gather}\label{eq:VstarU}
U\ast_\theta V=e^{2\pi i\theta} V\ast_\theta U .
\end{gather}
It is then the formal analogue of the smooth algebra of the noncommutative torus.

We can extend $\C\lbrak \hbar\rbrak $-linearly the map $f\mapsto\widetilde{f}$ in \eqref{eq:fxy} to a bijective map
\begin{gather*}
E_n:=A_n\lbrak\hbar\rbrak\to\widetilde{E}_n:=\mc{S}(\R\times\Z/n\Z)\lbrak\hbar\rbrak .
\end{gather*}
For $\theta,\theta'\in\Theta$,
we can def\/ine a left action of $A^{\theta'}_0$
on $\widetilde{E}_n$ and a right action of $A^\theta_0$ on $\widetilde{E}_n$ by
\begin{gather*}
a.\widetilde{f}:=\widetilde{a\ast_{\theta'}f} ,
\qquad
\widetilde{f}.a:=\widetilde{f\ast_\theta a} ,
\end{gather*}
for all $\widetilde{f}\in\widetilde{E}_n$ and $a\in A_0$.
A computation using
\eqref{eq:fxy} gives the explicit formulas (for $n\neq 0$):
\begin{subequations}\label{eq:bimod}
\begin{align}
(U.\widetilde{f})(x;k) &=\widetilde{f}\big(x-\tfrac{1}{n};k-1\big) ,&
(\widetilde{f}.U)(x;k) &=\widetilde{f}\big(x-\tfrac{1}{n}-\theta;k-1\big) , \\
(V.\widetilde{f})(x;k) &=e^{2\pi i(x-\frac{k}{n})}e^{-2\pi in\theta'x}\widetilde{f}(x;k) , &
(\widetilde{f}.V)(x;k) &=e^{2\pi i(x-\frac{k}{n})}\widetilde{f}(x;k) ,
\end{align}
\end{subequations}
where for a smooth function~$\psi$,
by $\psi(y+\theta')$ we mean the formal power series
\begin{gather*}
\psi(y+\theta'):=e^{\theta'\partial_y}\psi(y)=\sum_{k\geq 0}\tfrac{\theta'^k}{k!}\partial_y^k\psi(y) .
\end{gather*}
(If we replace $\theta'$ by a real number, although the above 	series is convergent only for $\psi$ analytic, the formulas \eqref{eq:bimod} are well def\/ined for any Schwartz function~$\widetilde{f}$.)

By comparing these formulas with equations~(2.1)--(2.5) of \cite{Pla06} we recognize the formal analogue of Connes--Rief\/fel imprimitivity bimodule $E_g(\theta)$, in the special case
\begin{gather}\label{eq:gn}
g= \begin{pmatrix} 1 & 0 \\ n & 1 \end{pmatrix} .
\end{gather}
Finally, for all $\widetilde f_1\in \widetilde E_{n_1}$ and $\widetilde f_2\in\widetilde E_{n_2}$,
one can compute $\widetilde{f_1\ast_\theta f_2}\in\widetilde E_{n_1+n_2}$
using~\eqref{eq:fxy} and f\/ind the explicit formula
\begin{gather}\label{eq:1aa}
\widetilde{f_1\ast_\theta f_2}(x;k)=\sum_{k_1+k_2=k}\!\widetilde{f}_1\big((1-n_2\theta)x+(1+n_1\theta)\tfrac{k_1n_2-k_2n_1}{n_1(n_1+n_2)}
;k_1\big)\widetilde{f}_2\big(x
-\tfrac{k_1n_2-k_2n_1}{n_2(n_1+n_2)};k_2\big) ,\!\!\!
\end{gather}
which is valid for $n_1$, $n_2$, $n_1+n_2\neq 0$. If, on the other hand, $f_1\in E_n$ and $f_2\in E_{-n}$ ($n\neq 0$), using the identity
\begin{gather*}
\psi(x)=\sum_{m\in\Z}e^{2\pi imx}\int_0^1
e^{-2\pi imq}\psi(q)\de q ,
\end{gather*}
valid for $\psi\in\mc{S}(\R)$,
we get
\begin{gather}\label{eq:1ab}
f_1\ast_\theta f_2=
\sum_{r,s\in\Z}U^rV^s \big(\widetilde f_1\big|\widetilde f_2.V^{-s}U^{-r}\big),
\end{gather}
where
\begin{gather*}
\big(\widetilde f_1\big|\widetilde f_2\big):=\sum_{k=1}^n\int_{-\infty}^{\infty}
\widetilde f_1\big((1+n\theta)x;k\big)\widetilde f_2(x;-k)\de x .
\end{gather*}
Equations \eqref{eq:1aa} and \eqref{eq:1ab} are a formal version of \cite[Proposition~1.2(a)]{PS02}.

We can now give an interpretation to Corollary~\ref{cor:3.3}: point~(2) is the analogue of, e.g., equation~(2.6) of \cite{Pla06}, stating that the algebra of endomorphisms of a f\/initely generated projective module over a noncommutative torus is another noncommutative torus with a dif\/ferent deformation parameter; point~(4) is the associativity of the pairing of bimodules in \cite[Proposition~1.2(b)]{PS02}.

\section{Complex structures and theta functions}\label{sec:5}

In order to include a complex structure in the construction, it is convenient to start from a~dif\/ferent realization of the Heisenberg group and of the principal bundle of previous section.

Let $\tau\in\C$ be a complex number with imaginary part $\Im(\tau)>0$ and
$\Lambda:=\Z+\tau\Z\subset\C$ a lattice.
We now construct a principal $U(1)$ bundle over the elliptic curve $E_\tau:=\C/\Lambda$, isomorphic to~$\T^2$ if one forgets about the complex structure.

We parametrize the Heisenberg group as follows: we set $H_3(\R):=\C\times\R$ with multiplication
\begin{gather}\label{eq:mult}
(z_1,t_1)\cdot(z_2,t_2)=\big(z_1+z_2,t_1+t_2+\Im(\bar z_1z_2)\big),
\end{gather}
where the bar denotes complex conjugation and $\Im$ the imaginary part.
Right invariant vector f\/ields are spanned by
\begin{gather}\label{eq:rivf}
\ga=\sqrt{\frac{\Im(\tau)}{\pi}}\left(\frac{\partial}{\partial z}
+\frac{i\bar z}{2}\frac{\partial}{\partial t}
\right)
,\qquad
\gb=\sqrt{\frac{\Im(\tau)}{\pi}}
\left(\frac{\partial}{\partial\bar z}
-\frac{iz}{2}\frac{\partial}{\partial t}
\right)
,\qquad
\gc=\frac{\Im(\tau)}{\pi i}\frac{\partial}{\partial t}
,
\end{gather}
where we identify elements of $\U(\mf{h}_3(\R))$ with their representation as dif\/ferential operators.

Let $H_3(\Z)\subset H_3(\R)$ be the subgroup generated by the elements $(1,0)$ and $(\tau,0)$.

\begin{rem}
For a more explicit description, one can verify that $H_3(\Z)$
coincides with the group
$G:=\big\{g_{m,n,k}:=\big(m+n\tau,(mn+2k)\Im(\tau)\big)\colon m,n,k\in\Z\big\}$. Indeed, from \eqref{eq:mult} we get
$g_{m,n,k}g_{m',n',k'}=g_{m+m',n+n',k+k'-nm'}$, proving that $G$ is a group.
Clearly $H_3(\Z)\subset G$, since $G$ contains the generators~$(1,0)$ and $(\tau,0)$ of~$H_3(\Z)$.
On the other hand, $g_{m,n,0}=(1,0)^m(\tau,0)^n$ and
$g_{0,0,1}=(1,0)\cdot (\tau,0)\cdot (-1,0)\cdot (-\tau,0)$
belong to $H_3(\Z)$, and since $g_{m,n,k}=g_{m,n,0}(g_{0,0,1})^k$,
this proves that $H_3(\Z)\supset G$.
Note that $H_3(\Z)$ is a proper subgroup of $\Lambda\times \Im(\tau)\Z$, since for example it doesn't contain the element $(0,\Im(\tau))$, and contains properly the group $2(\Lambda\times \Im(\tau)\Z)$.
\end{rem}

We def\/ine $M_3:=H_3(\R)/H_3(\Z)$.
As in previous section,
the action of central elements $(0,t)\in H_3(\R)$ descends to a principal action of $U(1)$ on $M_3$, and $M_3/U(1)\simeq E_\tau$.
In the notations of Section~\ref{sec:QUEA}, we set
$A:=C^\infty(M_3)$ and $A_n$ as in \eqref{eq:An}, so that
$A_0\simeq C^\infty(E_\tau)$.
The right invariance of $f\in A$ under the action of $(0,2\Im(\tau))\in H_3(\Z)$ proves that $f$ is periodic in $t$ with period $2\Im(\tau)$, hence
the dif\/ferential operator $\gc$ in~\eqref{eq:rivf} has integer spectrum (this explains the choice of normalization) and the subalgebra $A_\bullet=\bigoplus_{n\in\Z}A_n$
is dense in~$A$.

The map \eqref{eq:fxy} is replaced by the bijection \eqref{eq:fxyC} below.

\begin{prop}
For all $n\neq 0$ there is a bijection $A_n\ni f\mapsto\widetilde{f}\in\mc{S}(\R\times\Z/n\Z)$ given by
\begin{gather}\label{eq:fxyC}
f(z,t)=e^{\pi in\{\frac{1}{\Im(\tau)}t+xy\}}
\sum_{k\in\Z}e^{2\pi ikx}\widetilde{f}\big(y+\tfrac{k}{n};k\big),
\end{gather}
where the real coordinates $(x,y)\in\R^2$ are defined by $z:=x+\tau y$.
\end{prop}
\begin{proof}
This is essentially~\eqref{eq:fxy} modulo a reparametrization.
$f\in A$ if\/f it is right invariant under the action of the two
generators $(1,0)$ and $(\tau,0)$ of $H_3(\Z)$. In real coordinates,
we get the conditions
$f(x+1,y,t-y\Im(\tau))=f(x,y,t)$
and
$f(x,y+1,t+x\Im(\tau))=f(x,y,t)$.
As mentioned above, $f$ is also $2\Im(\tau)$-periodic in $t$, and
the condition $\gc f=nf$ says that every $f\in A_n$ is given by
$e^{\pi in\frac{1}{\Im(\tau)}t}$ times a function of~$x$,~$y$.
Def\/ine $F(x,y)$ implicitly by $f(x,y,z)=e^{\pi in\{\frac{1}{\Im(\tau)}t+xy\}}F(x,y)$.
The two invariance conditions above become $F(x+1,y)=F(x,y)$ and
$F(x,y+1)=F(x,y)e^{-2\pi inx}$. From the former, $F(x,y)=\sum\limits_{k\in\Z}e^{2\pi ikx}F_k(y)$ for some functions~$F_k$. The latter condition gives
$F_k(y+1)=F_{k+n}(y)$; if we def\/ine $\widetilde{f}(y;k):=F_k(y-\frac{k}{n})$,
the condition becomes $\widetilde{f}(y;k+n)=\widetilde{f}(y;k)$. Thus,
$\widetilde{f}(y;k)$ is periodic in $k$ with period~$n$.
Finally, from \cite[Lemma~3.2]{DFF13} it follows that $f$ is $C^\infty$
if\/f $\widetilde{f}$ is Schwartz.
\end{proof}

We can apply the same recipe of previous section, and deform the algebra $A$ with the twist $F_\theta=\exp\big\{\theta \ga \otimes\gb\big\}$, $\theta\in\Theta$.
The advantage is that now $A^\theta$ has, besides $A^\theta_0$, two additional associative subalgebras.
The coassociator
$\Phi_{\theta,\theta}=e^{\theta^2\ga \otimes \gc \otimes \gb}$, cf.~\eqref{eq:coassociator},
is $1$ on the kernels of $\ga$, $\gc$ and $\gb$. The second kernel is $A_0^\theta$, the f\/irst and third are related by a conjugation $z\mapsto\bar z$. We will focus on the latter.

As in previous section, $A_0^\theta$ is generated by two unitary functions
\begin{gather}\label{eq:UVtorus}
U(x,y,t):=e^{2\pi ix} ,\qquad
V(x,y,t):=e^{2\pi iy} .
\end{gather}
Since $U\ast_\theta V=e^{-\frac{\pi\theta\bar\tau}{\Im(\tau)}}UV$
and
$V\ast_\theta U=e^{-\frac{\pi\theta\tau}{\Im(\tau)}}VU$, we get the usual noncommutative torus commutation relation\footnote{The exchange
$U\leftrightarrow V$ in \eqref{eq:VUtorus} and \eqref{eq:UVtorus} is needed
to get the same commutation relation \eqref{eq:VstarU} and \eqref{eq:UstarV}.}
\begin{gather}\label{eq:UstarV}
U\ast_\theta V=e^{2\pi i\theta} V\ast_\theta U .
\end{gather}

Let $A_{\mathrm{hol}}:=\ker\gb$
and $A_{\mathrm{hol}}^\theta:=(A_{\mathrm{hol}}\lbrak\hbar\rbrak,\ast_\theta)$. Since $F_\theta$ is $1$ on $A_{\mathrm{hol}}\otimes A_{\mathrm{hol}}$, the product is undeformed: $a\ast_\theta b=ab\;\forall\;a,b\in A_{\mathrm{hol}}^\theta$ and $A_{\mathrm{hol}}^\theta$ is a \emph{commutative} associative subalgebra of $A^\theta$.

With an explicit computation we now check that elements of
$A_{\mathrm{hol}}$ are (essentially) classical theta functions on the torus.
Clearly $A_0\cap A_{\mathrm{hol}}=\C$ is the set of constant functions.
For $n\neq 0$, the set $A_n\cap A_{\mathrm{hol}}$ is described by the following lemma.

\begin{lem}
If $n<0$, $A_{\mathrm{hol}}\cap A_n=\{0\}$. If $n>0$, elements $f\in A_{\mathrm{hol}}\cap A_n$ are in bijection with elements $c_f\in\C^{\Z/n\Z}$ via the formula
\begin{gather}\label{eq:theta}
f(z,t)=e^{\pi in\{\frac{1}{\Im(\tau)}t+xy+\tau y^2\}}
\sum_{k\in\Z}c_f(k)q^{k^2/n}e^{2\pi ikz} ,
\end{gather}
where $q:=e^{\pi i\tau}$.
\end{lem}

\begin{proof}
With some algebraic manipulation we can rewrite \eqref{eq:fxyC} as follows
\begin{gather}\label{eq:rewrite}
f(z,t)=e^{\pi in\{\frac{1}{\Im(\tau)}t+xy+\tau y^2\}}
\sum_{k\in\Z}q^{k^2/n}e^{2\pi ikz}\left(e^{-\pi in\tau(y+\frac{k}{n})^2}
\widetilde{f}\big(y+\tfrac{k}{n};k\big)\right) .
\end{gather}
In real coordinates, $\gb$ is proportional to the dif\/ferential operator
$\nabla:=\tau\frac{\partial}{\partial x}-\frac{\partial}{\partial y}
+\Im(\tau)z\frac{\partial}{\partial t}$. Both the factor
$e^{\pi in\{\frac{1}{\Im(\tau)}t+xy+\tau y^2\}}$ and the holomorphic exponential $e^{2\pi ikz}$ are in the kernel of such an operator, and since
the product in parenthesis in \eqref{eq:rewrite} only depends on $y$:
\begin{gather*}
\nabla f(z,t)=-e^{\pi in\{\frac{1}{\Im(\tau)}t+xy+\tau y^2\}}
\sum_{k\in\Z}q^{k^2/n}e^{2\pi ikz}
\frac{\partial}{\partial y}\left(e^{-\pi in\tau(y+\frac{k}{n})^2}\widetilde{f}\big(y+\tfrac{k}{n};k\big)\right) .
\end{gather*}
One has $\nabla f=0$ if\/f
$e^{-\pi in\tau(y+\frac{k}{n})^2}\widetilde{f}\big(y+\tfrac{k}{n};k\big)=:c_f(k)$ does not depend on $y+\frac{k}{n}$, and in this case \eqref{eq:rewrite} reduces to
\eqref{eq:theta}. Since $\Im(\tau)>0$,
for $n>0$ the function $\widetilde f(y;k)=c_f(k)e^{\pi in\tau y^2}$ is of Schwartz class for any $c_f\in\C^{\Z/n\Z}$, while for $n<0$ it is of
Schwartz class only if it is zero.
\end{proof}

The series in \eqref{eq:theta} are the usual theta functions on $E_\tau$. If $n=1$, for example, the series in~\eqref{eq:theta} is proportional to the Jacobi's theta function
$\vartheta(z;q)=\sum\limits_{k\in\Z}q^{k^2/n}e^{2\pi ikz}$.

The algebraic structure of theta functions is encoded in the formula \eqref{eq:prodstar} below.

\begin{prop}
Let $n_1,n_2>0$.
For all $f_i\in A_{\mathrm{hol}}\cap A_{n_i}$, $i=1,2$,
\begin{gather}\label{eq:prodstar}
c_{f_1f_2}(k)=
\sum_{\substack{k_1,k_2\in\Z \\[1pt] k_1+k_2=k }}c_{f_1}(k_1)c_{f_2}(k_2)q^{\frac{1}{n_1n_2(n_1+n_2)}(k_1n_2-k_2n_1)^2} ,
\end{gather}
where $f\mapsto c_f$ is the map in \eqref{eq:theta}.
\end{prop}

\begin{proof}
A simple computation using \eqref{eq:theta} and the algebraic identity:
\begin{gather*}
\frac{k_1^2}{n_1}+\frac{k_2^2}{n_2}=
\frac{(k_1+k_2)^2}{n_1+n_2}+
\frac{(k_1n_2-k_2n_1)^2}{n_1n_2(n_1+n_2)}.\tag*{\qed}
\end{gather*}
 \renewcommand{\qed}{}
\end{proof}

\begin{rem}
In the $C^*$-algebraic setting,
for any f\/ixed modular parameter $\tau$, any $g$ as in \eqref{eq:gtheta} and $\theta\in\R$ solution of $g\theta=\theta$, a ring of ``quantum theta functions'' $B_g(\theta,\tau)$ can be def\/ined as a suitable ``holomorphic'' subalgebra of the tensor algebra $\bigoplus_{k\geq 0}E_g(\theta)^{\otimes k}$ (with tensor product over the algebra~$A_\theta$ of the noncommutative torus
and $E_g(\theta)^{\otimes 0}:=A_\theta$).
A product formula for quantum theta functions appeared f\/irst in~\cite{DS02} in terms of generators and relations (see also~\cite{PS02}), while an alternative formula which is closer to our notations is \cite[equation~(7.4)]{Vla06}.

For $g$ as in \eqref{eq:gn} and $n=1$, as one can easily check, the product $(c_{f_1},c_{f_2})\mapsto c_{f_1f_2}$ def\/ined by~\eqref{eq:prodstar} coincides with \cite[equation~(7.4)]{Vla06} and $B_g(\theta,\tau)\simeq A_{\mathrm{hol}}$
(while $B_g(\theta,\tau)$ is a subalgebra of~$A_{\mathrm{hol}}$ if $n>1$ in~\eqref{eq:gn}).
However, this is not surprising since, for~$g$ as in~\eqref{eq:gn}, the only solution to $g\theta=\theta$ is $\theta=0$.
The construction in this section, on the other hand, works in a formal setting and there is no constrain on~$\theta$. For any $\theta\in\Theta$, $A^\theta$ is well-def\/ined, although not associative, and has an associative (and commutative) subalgebra given by classical theta functions.
\end{rem}

\section{On higher rank vector bundles}\label{sec:6}
In this section, we describe how to derive (formal) imprimitivity bimodules associated to an arbitrary element of ${\rm SL}(2,\Z)$.

Let us f\/ix a $g$ as in \eqref{eq:gtheta}, with $d\neq 0$.
We denote by $\mathcal{M}_{c,d}\subset C^\infty(\R\times\T)$ the set of functions satisfying
\begin{gather}\label{eq:Mcd}
f(x+d,y)=e^{-2\pi icy}f(x,y) ,
\end{gather}
This is a $C^\infty(\T^2)$-right module with product given by pointwise multiplication.
One can verify that $\mathcal{M}_{c,d}$ is isomorphic to the module of (smooth) sections of a~rank~$|d|$ vector bundle on~$\T^2$ as follows.
Note that $\mathcal{M}_{c,d}\simeq\mathcal{M}_{-c,-d}$, so from now on we can assume that $d\geq 1$.

That $\mathcal{M}_{0,d}$ is a free module of rank~$d$ comes from the following observation.

\begin{rem}\label{rem:6.1}
By standard Fourier analysis, every smooth function $\varphi$ with period $d$ can be written (in a unique way) as
\begin{gather*}
\varphi(x)=\sum_{k=1}^de^{2\pi i\frac{k}{d}x}\varphi_k(x)
\end{gather*}
with $\varphi_k$ of period $1$. This gives a $C^\infty(\R/\Z)$-module isomorphism:
\begin{gather*}
C^\infty(\R/d\Z) \to C^\infty(\R/\Z)\otimes\C^d ,\qquad
\varphi\mapsto (\varphi_1,\ldots,\varphi_d) .
\end{gather*}
(Functions of period $d$ form a free module of rank $d$ over functions of period~$1$.)
\end{rem}

For the reason above, $\mathcal{M}_{0,d}\simeq \mathcal{M}_{0,1}\otimes\C^d$ as modules over $\mathcal{M}_{0,1}=C^\infty(\T^2)$.

From now on we forget about free modules and assume that $c\neq 0$ (and $d\geq 1$ as above). Note that the condition
$\det(g)=1$ in \eqref{eq:gtheta} guarantees that
$c$ and $d$ are coprime. Vice versa,
by B{\'e}zout's lemma such a $g$ exists for every coprime $c$, $d$ (although it is not unique).

\begin{lem}\label{lemma:unique}
Every $n\in\Z$ can be written, in a unique way, as $n=kc+md$ for some $1\leq k\leq d$ and $m\in\Z$.
\end{lem}

\begin{proof}
The map $\Z^2\ni (k,m)\to kc+md\in\Z$ is surjective, due to the identity
$n=(-nb)c+(na)d$ following from the determinant condition in \eqref{eq:gtheta}; since $(k,m)$
and $(k-d,m+c)$ have the same image, one can always choose $1\leq k\leq d$; restricted to
$[1,\ldots,d]\times\Z$ the map is also injective,
since $kc+md=k'+md'$~-- i.e., $(k-k')c+(m-m')d=0$~-- implies that $d$ must divide $k-k'$. But
$|k-k'|\leq d-1$, so it must be $k-k'=0$,
which also implies $m-m'=0$.
\end{proof}

Remark~\ref{rem:6.1} can then be rephrase as follows.

\begin{lem}
Every $\varphi\in C^\infty(\R/d\Z)$ can be written in a unique way as
\begin{gather*}
\varphi(x)=\sum_{k=1}^de^{2\pi i\frac{c}{d}kx}\varphi_k(x)
\end{gather*}
with $\varphi_1,\ldots,\varphi_d\in C^\infty(\R/\Z)$.
\end{lem}

\begin{proof}
From Lemma~\ref{lemma:unique},
we can write the Fourier series of $\varphi$ as
\begin{gather*}
\varphi(x)=
\sum_{n\in\Z}e^{2\pi i\frac{1}{d}nx}\hat\varphi_k
=\sum_{k=1}^de^{2\pi i\frac{c}{d}kx}\left(\sum_{m\in\Z}e^{2\pi imx}\hat\varphi_k\right)
\end{gather*}
and called $\varphi_k(x)=\sum\limits_{m\in\Z}e^{2\pi imx}\hat\varphi_k$ we get the desired result.
\end{proof}

\begin{prop}
$\mathcal{M}_{c,d}$ is isomorphic to the $C^\infty(\T^2)$-module of smooth sections of the vector bundle
\begin{gather}\label{eq:vbZ}
\begin{array}{@{}c}
\big(\T\times\R\times\C^d\big)/\Z
\\
\downarrow
\\
\T\times\R/\Z=\T^2,
\end{array}
\end{gather}
where the action of $\Z$ on $\T\times\R\times\C^d$ is generated by the map
\begin{gather*}
(x,y;v_1,v_2,\ldots,v_d)\mapsto
\big(x,y+1;e^{2\pi icx}v_d,v_1,\ldots,v_{d-1}\big) .
\end{gather*}
\end{prop}

\begin{proof}
Any $f\in\mathcal{M}_{c,d}$~-- since $e^{2\pi i\frac{c}{d}xy}f(x,y)$ is periodic in~$x$ with period~$d$~-- can be written (in a~unique way) in the form $f(x,y)=e^{-2\pi i\frac{c}{d}xy}\sum\limits_{k=1}^de^{2\pi i\frac{c}{d}kx}f_k(x,y)$, where $f_1,\ldots,f_d$ have period~$1$ in~$x$.
From \eqref{eq:Mcd} we get the conditions
\begin{gather*}
f_k(x,y+1)=f_{k-1}(x,y) ,\qquad\forall\, k\neq 1,
\qquad
f_1(x,y+1)=e^{2\pi icx}f_d(x,y) .
\end{gather*}
We may then think of $(f_1,\ldots,f_d)$ as a section of the vector bundle~\eqref{eq:vbZ}.
More precisely, the corresponding section is the map
\begin{gather*}
\T^2\ni [x,y]\mapsto \big[[x],y;f_1(x,y),\ldots,f_d(x,y)\big] \in \big(\T\times\R\times\C^d\big)/\Z.
\end{gather*}
The transformation $f\mapsto (f_1,\ldots,f_d)$ is the desired $C^\infty(\T^2)$-module isomorphism.
\end{proof}

Let $U$, $V$ be the generators of $C^\infty(\T^2)$ in \eqref{eq:VUtorus}, and $U'$, $V'$ the
$C^\infty(\T^2)$-linear endomorphisms of $\mathcal{M}_{c,d}$ given by
\begin{gather*}
(U'f)(x,y):=e^{2\pi iay}f(x+b,y) ,\qquad
(V'f)(x,y):=e^{2\pi i\frac{1}{d}x}f(x,y) .
\end{gather*}
(Here $a,b$ are the elements in the f\/irst row of~\eqref{eq:gtheta}.)
The operators $U'$, $V'$ are unitary if we equip~$\mathcal{M}_{c,d}$ with the inner product given by~$\inner{f_1,f_2}:=\int_{[0,d]\times [0,1]}f_1(x,y)^*f_2(x,y)\de x\hspace{1pt}\de y$,
and satisfy the def\/ining relation of the rational noncommutative torus with deformation parame\-ter~$b/d$, namely $U'V'=e^{2\pi i\frac{b}{d}}V'U'$.
Let us denote by $B_{b/d}$ the algebra of power series in $U'$, $V'$ with rapid decay coef\/f\/icients.

\looseness=-1
In order to apply the deformation machinery of Section~\ref{sec:3} we need an action of $\mf{h}_3(\R)$.
The pointwise product gives a map $\mathcal{M}_{c,d}\times\mathcal{M}_{c',d'}\to\mathcal{M}_{cd'+c'd,dd'}$. An action of $\mf{h}_3(\R)$ by derivations is
\begin{gather*}
\ga f=\frac{1}{\sqrt{2\pi}}\frac{\partial}{\partial x}f
,\qquad
\gb f=\frac{-i}{\sqrt{2\pi}}\left(\frac{\partial}{\partial y}+2\pi i\frac{c}{d}x\right)f
,\qquad
\gc f=\frac{c}{d}f
,
\end{gather*}
for all $f\in\mathcal{M}_{c,d}$.
The condition that $\gb$ is a derivation (i.e., satisf\/ies the Leibniz rule) f\/ixes the constant in front of the factor $x$, and consequently the normalization of $\gc$.
In particular $\mathcal{M}_{0,1}=C^\infty(\T^2)$ is in the kernel of~$\gc$.

A compatible action on $B_{b/d}\subset\operatorname{End}_{C^\infty(\T^2)}(\mathcal{M}_{c,d})$ is given by commutators: for all primitive $X\in\mf{h}_3(\R)$ and all $\xi\in B_{b/d}$, $X(\xi)$ is the endomorphism def\/ined by
$X(\xi)f :=X(\xi f)-\xi X(f)$ for all $f\in\mathcal{M}_{c,d}$.
With a simple computation one checks that $\gc$ is mapped to $0$, while $\ga\mapsto\frac{1}{\sqrt{2\pi}}\delta_{V'}$ and $\gb\mapsto\frac{-i}{\sqrt{2\pi}}\delta_{U'}$, where $\delta_{U'}$, $\delta_{V'}$ are the derivations def\/ined on generators by
\begin{gather*}
\delta_{U'}(U')=\tfrac{2\pi i}{d} U' ,\qquad
\delta_{U'}(V')=0 ,\qquad
\delta_{V'}(U')=0 ,\qquad
\delta_{V'}(V')=\tfrac{2\pi i}{d} V' .
\end{gather*}
We can now use the twisting element in Section~\ref{sec:3} to construct,
for any $\theta,\theta'\in\Theta$,
two new algebras $A_0^\theta:=(C^\infty(\T^2)\lbrak\nu\rbrak,\ast_\theta)$
and
$B_{b/d}^{\theta'}:=(B_{b/d}\lbrak\nu\rbrak,\ast_{\theta'})$
and to deform
$E_{c,d}:=\mathcal{M}_{c,d}\lbrak\nu\rbrak$ into a left
$B_{b/d}^{\theta'}$-module and right $A_0^\theta$-module.

Note that $A_0^\theta$ and $B_{b/d}^{\theta'}$ are both associative (they are in the kernel of~$\gc$, hence the coassociator is trivial), the former is generated by the two unitaries~$U$ and~$V$ with relation~\eqref{eq:VstarU}, the latter by the two unitaries $U'$ and $V'$ with relation
\begin{gather}\label{eq:ratdef}
U'\ast_{\theta'}V'=e^{2\pi i(d^{-2}\theta'+\frac{b}{d})}V'\ast_{\theta'}U' .
\end{gather}

The action $\alpha$ in \eqref{eq:nact} extends in an obvious way to an action of $\Q$ on $\Theta$, and from Lemma~\ref{gcl} we derive the following analogue of Proposition~\ref{prop:gcl}.

\begin{prop}
The left $B_{b/d}^{\theta'}$ action and right $A_0^\theta$ action
on $E_{c,d}$ commute, i.e.,
\begin{gather}\label{eq:genassB}
(\xi\ast_{\theta'}f_1)\ast_\theta f_2=
\xi\ast_{\theta'}(f_1\ast_\theta f_2)
\qquad \forall\,\xi\in B_{b/d}^{\theta'},\quad f_1\in E_{c,d},\quad f_2\in A_0^\theta ,
\end{gather}
if and only if $\theta'=\alpha_{\frac{c}{d}}(\theta)$.
\end{prop}

\begin{proof}
If $\theta'=\alpha_{\frac{c}{d}}(\theta)$,
since $\gc f=\frac{c}{d}f$ for all $f\in E_{c,d}$, we deduce that $\Phi_{\theta,\theta'}$ is the identity on $\xi\otimes f_1\otimes f_2$ and then~\eqref{eq:genassB} holds. On the other hand,
\begin{gather*}
m(m\otimes\id)(\Phi_{\theta,\theta'}-1)(V'\otimes f\otimes U)
=\big(e^{2\pi id^{-1}(1+\frac{c}{d}\theta)
\{\alpha_{c/d}(\theta)-\theta'\}}-1\big)V'fU
\end{gather*}
is zero for all $f\in E_{c,d}$ only if $\theta'=\alpha_{\frac{c}{d}}(\theta)$.
\end{proof}

Note that if $\theta'=\alpha_{\frac{c}{d}}(\theta)$, the parameter in~\eqref{eq:ratdef} is
\begin{gather*}
d^{-2}\theta'+\frac{b}{d}=\frac{a\theta+b}{c\theta+d}=g\theta
\end{gather*}
as expected.

The action of the two algebras on $E_{c,d}$ can be explicitly computed on generators.
The left action of~$U'$ and the right action of $V$ are undeformed, while for the remaining two generators one gets
\begin{gather*}
(f\ast_\theta U)(x,y) =e^{2\pi iy}f(x-\theta,y) , \\
(V'\ast_{\theta'} f)(x,y) =e^{2\pi i\frac{1}{d}x}e^{-2\pi icd^{-2}\theta'x}f\big(x,y-d^{-1}\theta'\big) .
\end{gather*}
It remains to compare what we obtained with the well-known formulas that one has in the $C^*$-algebraic setting, cf., e.g., equations~(2.1)--(2.5) of~\cite{Pla06}.
This can be done by means of the transform
$E_{c,d}\to\widetilde{E}_{c,d}:=\mc{S}(\R\times\Z/c\Z)\lbrak\hbar\rbrak$, $f\mapsto\widetilde{f}$, def\/ined by
\begin{gather*}
f(x,y)=\sum_{k\in\Z}\widetilde{f}\big(x+k\tfrac{d}{c};k\big)e^{2\pi iky}
\end{gather*}
and similar to \eqref{eq:fxy} (except for $\frac{1}{n}$ replaced by $\frac{d}{c}$).
Under this transform, the left/right module structure become:
\begin{alignat*}{3}
& \widetilde{U'\!\ast_{\theta'}\! f}(x;k) =\widetilde{f}\big(x-\tfrac{1}{c};k-a\big) ,\qquad &&
\widetilde{f\!\ast_\theta\! U}(x;k) =\widetilde{f}\big(x-\tfrac{d}{c}-\theta;k-1\big) , & \\
& \widetilde{V'\!\ast_{\theta'}\! f}(x;k) =e^{2\pi i\frac{1}{d}(x-k\frac{d}{c})}e^{-2\pi icd^{-2}\theta'x}\widetilde{f}(x;k) , \qquad &&
\widetilde{f\!\ast_\theta\! V}(x;k) =e^{2\pi i(x-k\frac{d}{c})}\widetilde{f}(x;k) ,&
\end{alignat*}
for all $\widetilde{f}\in\widetilde{E}_{c,d}$. These formulas reduce to
equations~(2.1)--(2.5) of~\cite{Pla06} if $\theta'=\alpha_{\frac{c}{d}}(\theta)$ and
after replacing formal deformation parameters by irrational numbers.

\pdfbookmark[1]{References}{ref}
\LastPageEnding

\end{document}